\documentclass[11pt, a4paper]{amsart}

\usepackage{amscd,verbatim}
\usepackage{amssymb}

\usepackage{bm}

\usepackage{amssymb, amsmath}
\usepackage[all]{xy}
\usepackage[inline]{enumitem}
\usepackage{hyphenat}
\usepackage[colorlinks,linkcolor=blue,citecolor=blue,urlcolor=red]{hyperref}
\usepackage[dvipsnames]{xcolor}
\usepackage{mathtools}
\usepackage{tikz-cd}
\usetikzlibrary{cd,decorations.pathreplacing}
\usepackage[labelformat=empty]{caption}
\usepackage{xfrac}
 \usepackage[ left=3cm, right=3cm, top = 2.8cm, bottom = 2.8cm ]{geometry}
\usepackage{stmaryrd}
\usepackage{bbm}

\makeatletter
\def\@tocline#1#2#3#4#5#6#7{\relax
  \ifnum #1>\c@tocdepth 
  \else
    \par \addpenalty\@secpenalty\addvspace{#2}%
    \begingroup \hyphenpenalty\@M
    \@ifempty{#4}{%
      \@tempdima\csname r@tocindent\number#1\endcsname\relax
    }{%
      \@tempdima#4\relax
    }%
    \parindent\z@ \leftskip#3\relax \advance\leftskip\@tempdima\relax
    \rightskip\@pnumwidth plus4em \parfillskip-\@pnumwidth
    #5\leavevmode\hskip-\@tempdima
      \ifcase #1
      \or\or \hskip 2em \or \hskip 2em \else \hskip 3em \fi%
      #6\nobreak\relax
    \dotfill\hbox to\@pnumwidth{\@tocpagenum{#7}}\par
    \nobreak
    \endgroup
  \fi}
\makeatother

\newcommand{\A}{\mathbf{A}}

\newcommand{\G}{\mathbf{G}}

\renewcommand{\P}{\mathbf{P}}

\newcommand{\Z}{\mathbb{Z}}
\newcommand{\F}{\mathbb{F}}

\newcommand{\sP}{\mathcal{P}}

\newcommand{\Cor}{\operatorname{\mathbf{Cor}}}

\newcommand{\Map}{\operatorname{map}}

\newcommand{\uHom}{\operatorname{\underline{Hom}}}

\newcommand{\Spec}{\operatorname{Spec}}

\newcommand{\pro}[1]{\text{\rm pro}_{#1}\text{\rm--}}
\newcommand{\tr}{{\operatorname{tr}}}

\newcommand{\eff}{{\operatorname{eff}}}

\newcommand{\Nis}{{\operatorname{Nis}}}
\newcommand{\et}{{\operatorname{\acute{e}t}}}

\newcommand{\Sym}{{\operatorname{Sym}}}

\renewcommand{\lim}{\operatornamewithlimits{\varprojlim}}
\newcommand{\colim}{\operatornamewithlimits{\varinjlim}}

\newcommand{\ol}{\overline}

\renewcommand{\phi}{\varphi}
\renewcommand{\epsilon}{\varepsilon}

\newcommand{\CI}{\operatorname{\mathbf{CI}}}

\newcommand{\bcube}{{\ol{\square}}}

\newcommand{\M}{\mathbf{M}}

\newcounter{spec}
{\end{list}}%

\newtheorem{lemma}{Lemma}[section]
\newtheorem{thm}[lemma]{Theorem}

\newtheorem{prop}[lemma]{Proposition}

\newtheorem{cor}[lemma]{Corollary}

\theoremstyle{definition}

\theoremstyle{remark}

\newtheorem{rmk}[lemma]{Remark}

\newtheorem{example}[lemma]{Example}

\newtheorem{nota}[lemma]{Notation}
\numberwithin{equation}{section}

\numberwithin{equation}{lemma}

\colorlet{LightRubineRed}{RubineRed!70!}

\usepackage[color = blue!20, bordercolor = black, textsize = tiny]{todonotes}

\usepackage{relsize} 
\usepackage[bbgreekl]{mathbbol} 
\usepackage{amsfonts} 
\DeclareSymbolFontAlphabet{\mathbb}{AMSb} 
\DeclareSymbolFontAlphabet{\mathbbl}{bbold} 

\DeclareSymbolFontAlphabet{\mathbbl}{bbold}

\def\lSm{\mathbf{lSm}}
\def\SmlSm{\mathbf{SmlSm}}
\def\Sm{\mathbf{Sm}}

\newcounter{elno}

\begin{document}

\def\THH{\operatorname{THH}}
\def\TC{\operatorname{TC}}
\def\TCmin{\operatorname{TC}^-}
\def\TP{\operatorname{TP}}
\def\HH{\operatorname{HH}}
\def\HC{\operatorname{HC}}
\def\HCmin{\operatorname{HC}^-}
\def\HP{\operatorname{HP}}

\def\Fil{\operatorname{Fil}}
\def\Gr{\operatorname{Gr}}
\def\gr{\operatorname{gr}}

\def\QSyn{\operatorname{QSyn}}
\def\QRSPerfd{\operatorname{QRSPerfd}}
\def\lQSyn{\operatorname{lQSyn}}
\def\lQRSPerfd{\operatorname{lQRSPerfd}}

\def\syn{\mathrm{syn}}
\def\Fsyn{\mathrm{Fsyn}}
\def\Fet{\mathrm{F\acute{e}t}}

\def\LogRec{\operatorname{\mathbf{LogRec}}}
\def\Ch{\operatorname{\mathrm{Ch}}}

\def\ltr{\mathrm{ltr}}

\def\kX{\mathfrak{X}}
\def\kY{\mathfrak{Y}}

\def\otCIsp{\otimes_{\CI}^{sp}}
\def\otCINissp{\otimes_{\CI}^{\Nis,sp}}

\def\tL{\tilde{L}}
\def\tX{\tilde{X}}
\def\tY{\tilde{Y}}
\def\tF{\widetilde{F}}
\def\tG{\widetilde{G}}

\def\Sh{\operatorname{\mathbf{Shv}}}
\def\PSh{\operatorname{\mathbf{PSh}}}
\def\Shltr{\operatorname{\mathbf{Shv}_{dNis}^{ltr}}}
\def\Shlog{\operatorname{\mathbf{Shv}_{dNis}^{log}}}
\def\Shvlog{\operatorname{\mathbf{Shv}^{log}}}
\def\Sm{\operatorname{\mathrm{Sm}}}
\def\SmlSm{\operatorname{\mathrm{SmlSm}}}
\def\lSm{\operatorname{\mathrm{lSm}}}
\def\FlSm{\operatorname{\mathrm{FlSm}}}
\def\FlQSm{\operatorname{\mathrm{FlQSm}}}
\def\FlQSyn{\operatorname{\mathrm{FlQSyn}}}
\def\lCor{\operatorname{\mathrm{lCor}}}
\def\SmlCor{\operatorname{\mathrm{SmlCor}}}
\def\PShltr{\operatorname{\mathbf{PSh}^{ltr}}}
\def\PShlog{\operatorname{\mathbf{PSh}^{log}}}
\def\logCI{\mathbf{logCI}} 

\def\Mod{\operatorname{Mod}}
\newcommand{\DM}[1][]{\operatorname{\mathcal{DM}}_{#1}}
\newcommand{\DMeff}[1][]{\operatorname{\mathcal{DM}}^{\eff}_{#1}}
\newcommand{\DA}[1][]{\operatorname{\mathcal{DA}}_{#1}}
\newcommand{\DAeff}[1][]{\operatorname{\mathcal{DA}}^{\eff}_{#1}}
\newcommand{\FDA}[1][]{\operatorname{\mathcal{FDA}_{#1}}}
\newcommand{\FDAeff}[1][]{\operatorname{\mathcal{FDA}^{\eff}_{#1}}}
\newcommand{\SH}[1][]{\operatorname{\mathcal{SH}_{#1}}}
\newcommand{\logSH}[1][]{\operatorname{\mathbf{log}\mathcal{SH}_{#1}}}
\newcommand{\logDA}[1][]{\operatorname{\mathbf{log}\mathcal{DA}_{#1}}}
\newcommand{\logDAeff}[1][]{\operatorname{\mathbf{log}\mathcal{DA}^{\eff}_{#1}}}
\newcommand{\logDM}[1][]{\operatorname{\mathbf{log}\mathcal{DM}_{#1}}}
\newcommand{\logDMeff}[1][]{\operatorname{\mathbf{log}\mathcal{DM}^{\eff}_{#1}}}
\newcommand{\logFDA}[1][]{\operatorname{\mathbf{log}\mathcal{FDA}_{#1}}}
\newcommand{\logFDAeff}[1][]{\operatorname{\mathbf{log}\mathcal{FDA}^{\eff}_{#1}}}
\newcommand{\WOmega}{\operatorname{\mathcal{W}\Omega}}
\def\Log{\operatorname{\mathcal{L}\textit{og}}}
\def\Rsc{\operatorname{\mathcal{R}\textit{sc}}}
\def\Pro{\mathrm{Pro}\textrm{-}}
\def\pro{\mathrm{pro}\textrm{-}}
\def\dg{\mathrm{dg}}
\def\plim{\mathrm{``lim"}}
\def\ker{\mathrm{ker}}
\def\coker{\mathrm{coker}}
\def\PrL{\mathcal{P}\mathrm{r^L}}
\def\PrLo{\mathcal{P}\mathrm{r^{L,\otimes}}}
\def\Spt{\mathcal{S}\mathrm{pt}}
\def\PSpt{\mathrm{Pre}\mathcal{S}\mathrm{pt}}
\def\Fun{\mathrm{Fun}}
\def\Sym{\mathrm{Sym}}
\def\CAlg{\mathrm{CAlg}}
\def\Poly{\mathrm{Poly}}
\def\Cat{\mathrm{Cat}}

\def\Alb{\operatorname{Alb}}
\def\bAlb{\mathbf{Alb}}
\def\Gal{\operatorname{Gal}}

\def\hofib{\mathrm{hofib}}
\def\fib{\mathrm{fib}}
\def\triv{\mathrm{triv}}
\def\ABl{\mathcal{A}\textit{Bl}}
\def\divsm#1{{#1_\mathrm{div}^{\mathrm{Sm}}}}

\def\cA{\mathcal{A}}
\def\cB{\mathcal{B}}
\def\cC{\mathcal{C}}
\def\cD{\mathcal{D}}
\def\cE{\mathcal{E}}
\def\cF{\mathcal{F}}
\def\cG{\mathcal{G}}
\def\cH{\mathcal{H}}
\def\cI{\mathcal{I}}
\def\cS{\mathcal{S}}
\def\cM{\mathcal{M}}
\def\cO{\mathcal{O}}
\def\cP{\mathcal{P}}

\def\tcA{\widetilde{\mathcal{A}}}
\def\tcB{\widetilde{\mathcal{B}}}
\def\tcC{\widetilde{\mathcal{C}}}
\def\tcD{\widetilde{\mathcal{D}}}
\def\tcE{\widetilde{\mathcal{E}}}
\def\tcF{\widetilde{\mathcal{F}}}

\def\one{\mathbbm{1}}

\def\XP{X \backslash \sP}
\def\M0a{{}^t\cM_0^a}
\newcommand{\Ind}{{\operatorname{Ind}}}

\def\Xkbar{\overline{X}_{\overline{k}}}
\def\dx{{\rm d}x}

\newcommand{\dNis}{{\operatorname{dNis}}}
\newcommand{\loget}{{\operatorname{l\acute{e}t}}}
\newcommand{\ket}{{\operatorname{k\acute{e}t}}}
\newcommand{\ABNis}{{\operatorname{AB-Nis}}}
\newcommand{\sNis}{{\operatorname{sNis}}}
\newcommand{\sZar}{{\operatorname{sZar}}}
\newcommand{\set}{{\operatorname{s\acute{e}t}}}
\newcommand{\cofib}{\mathrm{Cofib}}

\newcommand{\Gmlog}{\G_m^{\log}}
\newcommand{\Gmlogred}{\overline{\G_m^{\log}}}

\newcommand{\varcolim}{\mathop{\mathrm{colim}}}
\newcommand{\varlim}{\mathop{\mathrm{lim}}}
\newcommand{\tensor}{\otimes}

\newcommand{\eq}[2]{\begin{equation}\label{#1}#2 \end{equation}}
\newcommand{\eqalign}[2]{\begin{equation}\label{#1}\begin{aligned}#2 \end{aligned}\end{equation}}

\def\varplim#1{\text{``}\varlim_{#1}\text{''}}
\def\det{\mathrm{d\acute{e}t}}
\def\federem#1{\begin{color}{teal}{#1}\end{color}}
\def\tomrem#1{\begin{color}{purple}{#1}\end{color}}

\author{Alberto Merici}
\address{Institut f\"ur Mathematik, Universit\"at Heidelberg\\ MATHEMATIKON, Im Neuenheimer Feld 205, 69120  Heidelberg, Germany.}
\email[A. Merici]{merici@mathi.uni-heidelberg.de}

\thanks{A.M. is supported by Horizon Europe's Marie
Sk{\l}odowska-Curie Action PF 101103309 ``MIPAC''\\
MSC classes: 14F30 (Primary), 14F42, 19E15}
\title[Motivic p-adic tame cohomology]{Motivic $p$-adic tame cohomology}

\begin{abstract} 
We construct a comparison functor between ($\A^1$-local) tame motives and ($\bcube$-local) log-\'etale motives over a field $k$ of positive characteristic. This generalizes Binda--Park--{\O}stv{\ae}r's comparison for the Nisnevich topology. As a consequence, we construct an $E_\infty$-ring spectrum $H\Z/p^m$ representing mod $p^m$ tame motivic cohomology: the existence of this ring spectrum and the usual properties of motives imply some results on tame motivic cohomology and a comparison with log \'etale motivic cohomology.
\end{abstract}
\maketitle
\section{Introduction}
Let $k$ be a field of characteristic $p$ and let $\Sm_k$ (resp $\lSm_k$) denote the category of (log) smooth (log) schemes over $k$, and $\bcube$ the log scheme $(\P^1,\infty)$, which ideally sits in-between $\A^1$ and $\P^1$ and represents ``sections on $\A^1$ that behave well at $\infty$'' (see \cite[Figure 1]{BPOCras}). The triangulated category of effective logarithmic motives over a field $\logDMeff(k)$ and its non-effective counterpart $\logDM(k)$  were introduced in \cite{BPO} generalizing the $\A^1$-invariant category of Voevodsky motives of \cite{VoevTriangCat} (see also \cite{MVW}), by choosing a suitable version of the Nisnevich topology for log schemes (the dividing Nisnevich topology, $\dNis$) and localizing at $\bcube$, with the scope of studying non-$\A^1$-invariant cohomology theories of log schemes. Later, for $S$ any quasi-compact quasi-separated log scheme, the $\bcube$-homotopy category $\mathbf{log}\mathcal{H}(S)$ and the $S^1$ stable and $\P^1$-stable counterparts $\logSH^{S^1}(S)$ and $\logSH(S)$ were introduced by the same authors in \cite{BPO-SH}, generalizing Morel---Voevodsky $\A^1$-homotopy categories $\mathcal{H}(S)$, $\SH^{S^1}(S)$ and $\SH(S)$ of \cite{MV} (see also \cite{Ayoub-thesis}), in order to study cohomology theories of log schemes represented by sheaves of spectra. In \cite[Theorem 8.2.11]{BPO} and \cite[Theorem 4.4]{DoosungA1invlog}, it was shown that the functor $\omega\colon \lSm_k\to \Sm_k$ that sends a log scheme $(X,\partial X)$ to $X-|\partial X|$ induces fully faithful functors:\[
\DMeff(k)\xrightarrow{\omega^*} \logDMeff(k)\quad \SH^{S^1}(k)\xrightarrow{\omega^*} \logSH^{S^1}(k),
\]
both characterized by the fact that for all $G\in \logDMeff(k)$ (or $\logSH^{S^1}(k)$) and all $(\ol{X},\partial X)\in \lSm_k$\[
(\omega^*G)(X,\partial X) = G(X-|\partial X|).
\]
This justifies the choice of one notation to denote both functors: it will be clear from the context which one is considered.
The functors $\omega^*$ have right adjoints $\omega_*$ by design (the $\A^1$-colocalization). If $k$ satisfies resolutions of singularities as in Notation \ref{nota:RS} (analogous to \cite[Main Theorem I and II]{Hironaka}), then these functors send the motive of a smooth scheme $X$ to the log motive of any smooth log compactification $(\ol{X},\partial X)$ as in Remark \ref{rmk:comp}, so the functors $\omega_*$ are both characterized by the fact that for all $F\in \logDMeff(k)$ (or $\logSH^{S^1}(k)$) and all $X\in \Sm_k$ with smooth log compactification $(\ol{X},\partial X)$\[
(\omega_*)F(X) = F(\ol{X},\partial X).
\]
In \cite[Remark 6.3]{mericicrys}, it was shown that this adjunction cannot be promoted to an adjunction between \'etale and log \'etale motives: in fact, if $X\in \Sm_k$ and $Y\to X$ is an \'etale cover, then the log \'etale sheafification of the \v Cech nerve $L_{\loget}\omega^*M(Y^\bullet)\to L_{\loget}\omega^*M(X)$ need not be an equivalence: the counterexample (which comes from \cite{ertl2021integral}) is an Artin--Schreier cover, which has wild ramification.

In this article, we show that by substituting the \'etale topology with the tame topology defined by H\"ubner--Schmidt in \cite{HS2021}, we indeed have a positive result, namely:
    \begin{thm}[see Theorem \ref{thm:comparison}]\label{thm:mot-intro}
    Let $k$ be a perfect field that satisfies resolutions of singularities as in Notation \ref{nota:RS}. Then the adjunction\[
    \begin{tikzcd}
    \SH^{S^1}(k) \ar[r,"\omega^*", shift left = 2] &\logSH^{S^1}(k)\ar[l,"\omega_*"]
    \end{tikzcd}
    \]
    induces an adjunction\[
    \begin{tikzcd}
    \SH^{S^1}_t(k) \ar[r,"\omega^*_t", shift left = 2] &\logSH^{S^1}_{\loget}\ar[l,"\omega_*^t"](k)
    \end{tikzcd}
    \]
fitting in commutative diagrams:\[
    \begin{tikzcd}
    \SH^{S^1}(k) \ar[r,"\omega^*"]\ar[d,"L_{(\A^1,t)}"] &\logSH^{S^1}(k)\ar[d,"L_{(\bcube,\loget)}"] && 
    \SH^{S^1}(k) &\logSH^{S^1}(k)\ar[l,"\omega_*"]\\
     \SH^{S^1}_t(k) \ar[r,"\omega^*_t"] &\logSH^{S^1}_{\loget}(k) &&
    \SH^{S^1}_t(k) \ar[u,"i_{(\A^1,t)}"]&\logSH^{S^1}_{\loget}\ar[l,"\omega_*^t"](k)\ar[u,"i_{(\bcube,\loget)}"]
    \end{tikzcd}
\]
and similarly for $\DMeff$.
\end{thm}
This result should be compared to \cite{AHI-Atiyah}, where the authors also consider the $\A^1$-colocalization of the motivic spectra constructed in \cite{AnnalaIwasaUnivers}, and get an inclusion $\SH(k)\to \mathbf{Mod}_{1_{\A^1}}\mathcal{MS}(k)$ that sends the motive of $X$ to the total fiber of the Gysin map induced the inclusion of the boundary $|\partial X|\hookrightarrow \ol{X}$ of any smooth log compactification $(\ol{X},\partial X)$ of $X$ as in Remark \ref{rmk:comp} (see \cite[Remark 6.26]{AHI-Atiyah}). For the same reasons explained in \cite[Remark 6.3]{mericicrys}, this cannot be promoted to the \'etale version of $\mathcal{MS}$, and our result shows a promising first step in promoting this in the tame setting. 
We remark that, on the other hand, the motivic spectra of \cite{AnnalaIwasaUnivers} are intrinsically $\P^1$-stable, while the result of Theorem \ref{thm:mot-intro} holds already in an $S^1$-stable setting.

One main application of the previous theorem is the following: for a fixed $m$, let $\nu_m(i)$ denote the mod $p^m$ motivic sheaves of Bloch--Illusie--Milne (to avoid confusion, we will not refer to them as the logarithmic de Rham--Witt sheaves, as they are not sheaves on logarithmic schemes): they are strictly $\A^1$-invariant Nisnevich sheaves with transfers (so $\nu_m(i)[0]\in \DMeff(k)$), but not strictly $\A^1$-invariant \'etale sheaves. In \cite{mericicrys}, we showed that the cohomology of the Rham--Witt sheaves with log poles $W_m\Lambda^n$ of Hyodo--Kato (see \cite{mokrane} or \cite{matsuue}) are representable in the category $\logDMeff[\loget](k)$, and therefore $L_{\loget}\omega^*(\nu_m(i)[0])\simeq \widetilde{\nu_m(i)}[0]$ (see Example \ref{ex:drW}), where $\widetilde{\nu_m(i)}$ is the log \'etale sheaf
\begin{equation}\label{eq:nu-tilde}
\widetilde{\nu_m(i)}\colon (X,\partial X)\mapsto \nu_m(i)(X-|\partial X|).
\end{equation}
As a consequence, we immediately deduce that the presheaves $H^1_t(-/k,\nu_m(n))$ are $\A^1$-invariant (see Example \ref{ex:drW}). Moreover, we deduce the $\A^1$-invariance and $\P^1$-stability of higher tame motivic cohomology by the purity result of Koubaa \cite[Theorem 1.3.1]{amine}, with the same assumptions (RS1) and (RS2) as in \ref{nota:RS}. Putting everything together we show that
\begin{thm}[see \eqref{eq:A1-inv}, Lemma \ref{lem:pbf} and Remark \ref{rmk:GL}]\label{thm:motivic-dRW} 
Let $k$ be a perfect field of characteristic $p$ that satisfies resolutions of singularities as in \ref{nota:RS}. For all $m$, the object \[\{\nu_m(i)[-i]\}_{i\geq 0}\in \mathrm{GrCAlg}(\cD(\Sh_{t}(\Sm_k,\Z/p^m)))\] builds up to an $E_\infty$-ring spectrum $H\Z/p^m$ in $\DM[t](k,\Z/p^m)$ such that for all $X\in \Sm_k$ we have that\[
\Map(\Sigma^{\infty}(X), \Sigma^{p,q}H\Z/p^m) \simeq R\Gamma_t(X,\nu_m(q))[p-q]
\]
Moreover, $H\Z/p^m$ is the unit of $\DM[t](k,\Z/p^m)$. 
\end{thm}
The choice of the terminology $H\Z/p^m$ is justified by \cite{GL}, in fact \emph{a posteriori} $H\Z/p^m$ is the image of motivic cohomology via the localization $\DM(k,\Z/p^m)\to \DM[t](k,\Z/p^m)$.  Using the motivic properties of $\DM$ (see Remark \ref{rmk:pbf-gysin}), we deduce immediately the result:
\begin{thm}\label{thm:main-intro} Let $k$ be a perfect field of characteristic $p$ that satisfies resolutions of singularities as in \ref{nota:RS}. For all $m\geq 1$ and $q\geq 0$, for $X\in \Sm_k$, $\cE\to X$ a vector bundle of rank $r+1$, then we have an isomorphism
        \[
        H^q_t(\cE/k,\nu_m(n))\cong H^q_t(X/k,\nu_m(n))
        \]
        and if $\P(\cE)$ is the associated projective bundle, the Chern classes induce an isomophism\[
        H^q_t(\P(\cE)/k,\nu_m(n))\cong \bigoplus_{i=0}^r H^{q-i}_t(X/k,\nu_m(n-i)).
        \]
\end{thm}
The case $n=0$ of Theorem \ref{thm:motivic-dRW} has already been proved in \cite[15.4]{HS2021} (still assuming that $k$ satisfies resolutions of singularities), deducing it from the adic versions \cite[Corollary 14.5 and 14.6]{Huebner2021}. 
Finally, another interesting application of our result is the following comparison
\begin{thm}\label{thm:comparison-logetale}
    Let $k$ be a field of characteristic $p$ satisfying resolutions of singularities as in \ref{nota:RS}. For all $X\in \Sm_k$ with smooth log compactification $(\ol{X},\partial X)$ as in Remark \ref{rmk:comp} we have that\[
R\Gamma_t(X,\nu_m(n))\simeq R\Gamma_{\loget}((\ol{X},\partial X),\widetilde{\nu_m(n)}),
\]
where $\widetilde{\nu_m(n)}$ is as in \eqref{eq:nu-tilde}.
\end{thm}
We remark that this result requires a very special property of $\nu_m(n)$, \emph{i.e.} the fact that $L_{\loget}\omega^*(\nu_m(n)[0])\simeq \widetilde{\nu_m(n)}[0]$ is $(\bcube,\loget)$-local: we do not know whether to expect this to hold in general for any $(\A^1,\Nis)$-local object.
\subsection*{Future perspectives} 

We expect Theorem \ref{thm:mot-intro} to hold without the assumption on resolutions of singularities. In fact, it was shown in \cite{BLMP} that log prismatic and syntomic cohomology give rise to motivic spectra in $\logSH_{\loget}(S)$ for any quasisyntomic $p$-adic scheme $S$. We expect a similar result as Theorem \ref{thm:mot-intro} to hold for a general base scheme $S$, so that the object $\omega_*^t\mathbf{E}^{\syn}$ can be studied as a tame version of syntomic cohomology. 

\subsection*{Acknowledgements}
	The author would like to thank F. Binda, D. Park, T. Lundemo, K. H\"ubner, A. Schmidt, S. Saito and P.A. {\O}stv{\ae}r for many valuable discussions and comments on earlier versions of the results in this paper, and Joseph Ayoub for pointing out a gap in a previous version of Section 3. He also thanks A. Koubaa for sharing his version of \cite{amine}. He also thanks the anonymous referee for a meticulous analysis of the paper, providing helpful comments which filled some small gaps in the arguments and led to an improved presentation. This project is supported by the MSCA-PF \emph{MIPAC} carried out at the University of Milan. The author is very thankful for the hospitality and the great work environment.

\section{Tame and log-\'etale motives}

We recall the definition of the tame site of \cite{HS2021}. For $S$ a scheme and $X$ an $S$-scheme, the tame site $(X/S)_t$ was defined as the site whose underlying category is $X_{\et}$, and $\{U_i\to U\}$ is a tame cover if an only if for all $x\in U$ and all $S$-valuations $v$ on $k(x)$, there is $y\in U_i$ for some $i$ lying over $x$ and a valuation $w$ on $k(y)$ extending $v$ such that the extension of valued fields $(k(x),v)\to (k(y),w)$ is tame, \emph{i.e.} the ramification index of $\cO_v\to \cO_w$ is prime to the characteristic of the residue field of $\cO_v$. In the rest of the paper, we will only consider the case $S=\Spec(k)$ where $k$ is a perfect field of characteristic $p$, therefore to ease the notation, for $X$ a $k$-scheme we will write $H^*_t(X,-)$ for $H^*_t(X/k,-)$.

We recall the most important properties of tame cohomology:
\begin{enumerate}
    \item By construction, every tame sheaf is a Nisnevich sheaf
    \item If $X$ is a quasi-compact $k$-scheme every tame cover admits a finite subcover \cite[Theorem 4.1]{HS2021}, therefore it is enough to consider covers $U\to X$ where $U=\sqcup U_i$ for a finite cover $\{U_i\to X\}$.
    \item On quasi-compact quasi-separated $k$-schemes, tame cohomology commutes with filtered colimits of sheaves \cite[Theorem 4.5]{HS2021} and cofiltered limits of schemes with affine transition maps \cite[Theorem 4.7]{HS2021}
    \item If $F$ is an \'etale sheaf of $\Z/m\Z$-modules where $m$ is prime to $p$, then $H^q_t(X,F)=H^q_{\et}(X,F)$ for every $X$ quasi-compact $k$-scheme \cite[Proposition 8.1]{HS2021}.
    \item If $X$ is proper over $k$ with the property that every finite set of points is contained in an affine open, then for every tame sheaf $F$ of abelian groups $H^q_t(X,F) = H^q_{\et}(X,a_{\et}F)$ \cite[Proposition 8.2]{HS2021}.
\end{enumerate}

Let $\Cor_k$ be the category of finite correspondences over $k$. By \cite[Lemma 16.1]{HS2021} the category of tame sheaves with transfers $\Sh_t(\Cor_k)$ is a Grothendieck abelian category and for every $F\in \Sh_t(\Cor_k)$ the tame cohomology presheaves \[
U\mapsto H^q_t(U,F)
\]
are presheaves with transfers. Therefore, similarly to Voevodsky's construction, one gets the stable $\infty$-category of effective tame motives and of tame motivic sheaves of $S^1$-spectra as\[
\DMeff[t](k):= L_{(\A^1,t)}\cD(\Sh_t(\Cor_k))\quad \SH_t^{S^1}(k):= L_{(\A^1,t)}\Sh_t(\Sm_k,\Spt).
\]
As usual, the graph functor $\Sm_k\hookrightarrow \Cor_k$ and the Dold-Kan functor $\cD(\Z)\to \Spt$ induce a localization\[
\SH^{S^1}_t(k)\to \DMeff[t](k)
\] 
and there are evident localizations\[
\begin{tikzcd}
\DMeff[\Nis](k)\ar[r,"L_{(\A^1,t)}"]&
\DMeff[t](k)\ar[r,"L_{(\A^1,\et)}"]&\DMeff[\et](k)\\
\SH^{S^1}_{\Nis}(k)\ar[r,"L_{(\A^1,t)}"]&
\SH^{S^1}_t(k)\ar[r,"L_{(\A^1,\et)}"]&\SH^{S^1}_{\et}(k)\\
\end{tikzcd}
\]
with right adjoints\[
\begin{tikzcd}
\DMeff[\et](k)\ar[r,"i_{(\A^1,\et)}"]&
\DMeff[t](k)\ar[r,"i_{(\A^1,t)}"]&\DMeff[\Nis](k)\\
\SH^{S^1}_{\et}(k)\ar[r,"i_{(\A^1,\et)}"]&
\SH^{S^1}_t(k)\ar[r,"i_{(\A^1,t)}"]&\SH^{S^1}_{\Nis}(k)
\end{tikzcd}
\]
For $X\in \Sm_k$ and $\tau$ any of the topologies above, we let $\cM^\tau(X)\in \DMeff[\tau](k)$ and $\cH^\tau(X)\in \SH^{S^1}_{\tau}(k)$ the image of the Yoneda funtcor. By sheafifying \cite[Theorem 3.2.23]{MV} we get a fiber sequence\[
\cH^t(X-Z)\to \cH^t(X)\to \cH^t(\mathrm{Th}_{N_Z}).
\]
\begin{rmk}\label{rmk:pbf-gysin}
    In the case with transfers, by \cite[Properties 14.5.]{MVW} we deduce that
\begin{enumerate}
    \item \label{item:1} If $\cE \to X$ is a vector bundle of rank $r+1$, we have an equivalence\[
        \cM^t(\cE)\xrightarrow{\simeq} \cM^t(X),
    \]
    and if $\P(\cE)$ is the associated projective bundle, the Chern classes induce an equivalence\[
    \bigoplus_{i=0}^r \cM^t(X)\otimes \cM^t(\P^1,i_0)\simeq \cM^t(\P(\cE))
    \]
    where $\cM^t(\P^1,i_0)$ is the complement of the splitting of $\cM^t(\P^1)\to \cM^t(k)$ given by the zero section.
    \item \label{item:2} Let $X$ be a smooth scheme over $k$ and $Z$ a smooth closed
subscheme of $X$ of codimension $c$. Then there is a Gysin fiber sequence\[
\cM^t(X-Z)\to \cM^t(X)\to \cM^t(Z)\otimes \cM^t(\P^1,i_0)^{\otimes c}
\]
\end{enumerate}
\end{rmk}

    Moreover, by \cite[Proposition 15.7]{MVW}, $\P^1$ is a symmetric object, so we can construct the stable $\infty$-category of tame motives (resp. tame motivic spectra) as\[
    \DM[t](k):=\Spt_{\P^1}(\DMeff[t](k))\quad \SH_t(k):=\Spt_{\P^1}(\SH^{S^1}_t(k))
    \]
using the general machinery of \cite{robalo} and \cite{HoveySpectra} (see also \cite[Section 1]{AnnalaIwasaUnivers}). In particular, we have the following result, whose proof is completely formal (see \emph{e.g.} \cite[\S 6.23]{CiDeg-LocalStableHomAlg}):
\begin{thm}\label{thm:ring-spectra}
Let $E_*$ be a graded commutative monoid in $\SH^{S^1}_t(k)$, together with a section $c\colon \cH^t(\P^1)\to E_1[2]$ such that for all $X \in \Sm_k$ and all $i$, the following composition is an equivalence:
\[
        \begin{tikzcd}
        \Map(\cH^t(X),E_i)\ar[r]\ar[ddrr,bend right=8] &\Map(\cH^t(\P^1_X),E_i\otimes \cH^t(\P^1)) \ar[r,"c"] &\Map(\cH^t(\P^1_X),E_i\otimes E_{1}[2])\ar[d,"\mu_{i,1}"]\\
        &&\Map(\cH^t(\P^1_X),E_{i+1}[2])\ar[d]\\
        &&\Map(\cH^t(\P^1_X,i_0),E_{i+1}[2]).
        \end{tikzcd}
\]
Then there is $\mathbf{E}\in \CAlg(\SH_t(k))$ such that for all $X\in \Sm_k$\[
\Map_{\SH_t(k)}(\Sigma^{\infty}(\cH^t(X)),\Sigma^{m,n}\mathbf{E}) \simeq R\Gamma_t(X,E_n[m]).
\]
Similarly, if $E_*$ is a graded commutative monoid in $\DMeff[t](k)$ as above, then there is a ring spectrum $\mathbf{E}$ in $\DM[t](k)$ representing $E_*$.
\end{thm}
By construction, if $\mathbf{E}\in \CAlg(\DM[t](k))$ representing $E_*$, then for all $q\geq 0$ the properties listed above imply:
\begin{enumerate}
    \item If $\cE \to X$ is a vector bundle of rank $r+1$, we have\[
        H^q_t(\cE,E_n)\xrightarrow{\simeq} H^q_t(X,E_n).
    \]
    and\[
    \bigoplus_{i=0}^r H^{q-2i}_t(X,E_{n-i})\cong H^q_t(\P(\cE),E_n).
    \]
    \item Let $X$ be a smooth scheme over $k$ and $Z$ a smooth closed
subscheme of $X$ of codimension $c$. Then the Gysin sequence gives\[
H^q_{t,Z}(X,E_n)\cong H^{q-2c}_{t}(Z,E_{n-c}).
\]
\end{enumerate}
We now compare tame with \'etale motives with invertible coefficients. First we record the following result, which is essentially \cite[Proposition 8.1]{HS2021}.
\begin{prop}\label{prop:tame-etale-transfers}
    Let $S$ be an $\F_p$-scheme, $X$ an $S$-scheme and $F$ be an \'etale sheaf of $\Z[1/p]$-modules on $X$. Then $H^q_{\et}(X,F)=H^q_t(X/S,F)$.  
\end{prop}
\begin{proof}
  The same proof of \emph{loc.cit.} works here: indeed if $R$ is henselian with residue field $k(R)$, then for all \'etale sheaves $G$ (not necessarily torsion) on $\Spec(R)$ we have $H^q_{\et}(\Spec(R),G) = H^q_{\et}(\Spec(k(R)),G_{k(R)})$: this is well known, see \cite[VII, Cor 8.6]{SGA4}. Then this allows to conclude that for every tame point $(\ol{x}_L,\ol{v}_L)$ of $X$ as in \emph{loc.cit.}\[
    H^q_{\et}(X_{(\ol{x}_L,\ol{v}_L)}^{th},F)\cong  H^q_{\et}(\Spec(L),F_{L}) = 0,
    \]
    since the Galois group of $L$ is a pro-$p$-group.
\end{proof}
\begin{lemma}\label{lem:tame-etale-sheafification}
Let $F$ be a tame sheaf of $\Z[1/p]$-modules with transfers. Then $F$
is also an \'etale sheaf with transfers.
    \begin{proof}
        The proof is completely analogous to \cite[Lemma 14.21]{MVW}, we report it for completeness. As in \emph{loc.cit.} we need to show that the kernel and cokernel of $F\to F_{\et}$ vanish, hence we may suppose that $F_{\et}=0$. If $F\neq 0$, there is $S=\Spec(R)\in \widetilde{\Sm_k}$ with $R$ tamely henselian and a nonzero element $c\in F(S)$. Since $F_{\et}=0$, there is a finite \'etale map $S'\to S$ with $c_{S'}=0$. Since $S$ is tamely henselian, the degree of $S'\to S$ is $p^r$ for $r\geq 0$ as the absolute Galois group of $k(S)$ is a pro-$p$-group, so since the map\[
        F(S)\to F(S')\to F(S)
        \]
        is the multiplication by $p^r$, which is invertible in $F(S)$, we conclude that $c=0$. This contradiction shows that $F=0$. 
    \end{proof}
\end{lemma}
\begin{cor}
    The sheafification map $\cD(\Sh_{t}^{tr}(k,\Z[1/p])\to \cD(\Sh_{\et}^{tr}(k,\Z[1/p])$ is an equivalence, therefore $\DMeff[t](k,\Z[1/p])\simeq \DMeff[\et](k,\Z[1/p])$ and $\DM[t](k,\Z[1/p])\simeq \DM[\et](k,\Z[1/p])$.
    \begin{proof}
        This is analogous to \cite[Proposition 2.2.9]{CD-etale} and follows from Proposition \ref{prop:tame-etale-transfers} and Lemma \ref{lem:tame-etale-sheafification}.
    \end{proof}
\end{cor}

Recall now the effective log motives $\logDMeff(k)$ and $S^1$-spectra $\logSH^{S^1}(k)$ defined in \cite{BPO} and \cite{BPO-SH}, with the localizations\[
\begin{tikzcd}
\logDMeff[\dNis](k)\ar[r,"L_{(\bcube,\loget)}"]&\logDMeff[\loget](k)\\
\logSH^{S^1}_{\dNis}(k)\ar[r,"L_{(\bcube,\loget)}"]&\SH^{S^1}_{\loget}(k).\\
\end{tikzcd}
\] 
We let $\cH_{\log}^{\tau}(-)$ and $\cM_{\log}^{\tau}(-)$ be the Yoneda functors similary as before. There are adjoint functors (see \cite[Construction 4.0.8]{BPO-SH})\[
\begin{tikzcd}
    \SH^{S^1}(k) \ar[r,"\omega^*"] &\logSH^{S^1}(k)\ar[l,shift left = 2,"\omega_*"]\ar[l,"\omega_\sharp"', shift right = 3]
\end{tikzcd}
\]
and similarly for $\DM$. Recall the assumption:
	\begin{nota}\label{nota:RS}
		We say that $k$ satisfies resolutions of singularities if the following two properties are satisfied (see \cite[Definition 7.6.3]{BPO} and \cite[Main Theorem I and II]{Hironaka}):
		\begin{enumerate}
			\item[(RS1)] For any integral scheme $X$ of finite type over $k$, there is a proper birational morphism $Y \to X$ of schemes over $k$, which is an isomorphism on the smooth locus, such that $Y$ is smooth over $k$.
			\item[(RS2)] Let $f \colon Y \to X$ be a proper birational morphism of integral schemes over $k$ such that $X$ is smooth over $k$ and let $Z_1, \ldots , Z_r$ be smooth divisors forming a strict normal crossing divisor on $X$. Assume that
			\[f^{-1}(X-Z_1 \cup\ldots \cup Z_r)\to X-Z_1 \cup\ldots \cup Z_r\] is an isomorphism. Then there is a sequence of blow-ups\[
			X_n\xrightarrow{f_{n-1}} X_{n-1} \xrightarrow{f_{n-2}}\ldots \xrightarrow{f_{0}}X_0\simeq X
			\]
			along smooth centers $W_i \subseteq X_i$ such that
			\begin{enumerate}[label=\alph*.]
				\item the composition $X_n \to X$ factors through $f$,
				\item $W_i$ is contained in the preimage of $Z_1 \cup \ldots \cup Z_r$ in $X_i$,
				\item $W_i$ has strict normal crossing with the sum of the reduced strict transforms of \[
				Z_1,\ldots,Z_r,f_0^{-1}(W_0),\ldots,f_{i-1}^{-1}(W_{i-1})\]
				in $X_i$.
			\end{enumerate}
		\end{enumerate}
	\end{nota}
    \begin{rmk}\label{rmk:comp}
        If $k$ satisfies (RS1) and (RS2) above, then for every scheme $X\in \Sm_k$ there exists $\ol{X}\in \Sm_k$ with an open immersion $j\colon X\subseteq \ol{X}$ such that $X-j(X)$ is the support of a simple normal crossing divisor. This divisor supports then a DF log structure $\partial X$ in the usual way (see \cite[Definition 1.7.1]{ogu}). The resulting log scheme $(X,\partial X)$ is called a \emph{smooth log compactification}.
    \end{rmk}
By \cite[Theorem 4.4 (2)]{DoosungA1invlog} (resp.  \cite[Proposition 8.2.8]{BPO}), if $k$ satisfies (RS1) and (RS2) then for all $X\in \Sm_k$ with smooth log compactification $(\ol{X},\partial X)$ as above, the image of $\cH^{\Nis}(X)$ (resp. $\cM^{\Nis}(X)$) via the above functor is $\cH_{\log}^{\dNis}(\ol{X},\partial X)$ (resp. $\cM_{\log}^{\dNis}(\ol{X},\partial X)$), in particular the functor $\omega^*$ is monoidal. 
\begin{lemma}\label{lem:tame-descent}
Let $k$ be a perfect field of characteristic $p$ that satisfies resolutions of singularities as in \ref{nota:RS}. Let $U\to X$ be a tame cover in $\Sm_k$ and for all $m$ let $U^m$ denote the $m$-fold fiber product $U\times_X U\ldots \times_X U$. Then the colimit along the \v Cech nerve induces an equivalence:\[
L_{(\bcube,\loget)}\omega^*\cH^{\Nis}(X)\simeq \colim_m L_{(\bcube,\loget)}\omega^*\cH^{\Nis}(U^m)\quad \textrm{in }\logSH^{S^1}_{\loget}(k)
\]
and similarly for $\logDMeff[\loget]$.
\end{lemma}
\begin{proof}
    We only do the proof for $\SH$, the proof of $\DM$ is similar (in fact easier). If $U\to X$ is finite, then this is \cite[Lemma 5.4]{mericicrys}. In general, we proceed by induction on the dimension of $X$. If the dimension is zero, then $U\to X$ is finite and we conclude. Since $U$ and $X$ are smooth separated of finite type over $\Spec(k)$, $U\to X$ is locally of finite presentation and separated. Since $U\to X$ is \'etale, it is flat and locally quasi-finite (see \cite[\href{https://stacks.math.columbia.edu/tag/03WS}{Tag 03WS}]{stacks-project}), and since $U$ is quasi-compact it has universally bounded fibres by \cite[\href{https://stacks.math.columbia.edu/tag/03JA}{Tag 03JA}]{stacks-project}, so by \cite[\href{https://stacks.math.columbia.edu/tag/07RY}{Tag 07RY}]{stacks-project} there is a closed subscheme $Z\subseteq X$ such that $U_{|X-Z}\to X-Z$ is finite (we consider $Z_{n-1}$ in \emph{loc.cit.}). Let $d$ be the relative codimension of $Z$ in $X$ and let $Z_0\subset Z$ be the singular locus of $Z$, so that $Z-Z_0\subseteq X-Z_0$ is a closed immersion in $\Sm_k$. Let $V\to X-Z_0$ be an trivializing open for the normal bundle of $Z-Z_0\subseteq X-Z_0$: by taking $V$ small enough, we can suppose that $(Z-Z_0)$ has pure codimension $d$ inside $V$.
    Let $V':= V\times_X(Z-Z_0)$ and $U' = U\times_X V\times_X(Z-Z_0)$. Then since $U^m\times_X V\to V$ and $V\to X$ are all flat, the normal bundle commute with the pullback so the Gysin sequences imply that the columns of the following diagram are fiber sequences:\[
    \begin{tikzcd}
\colim_m L_{(\bcube,\loget)}\omega^*\cH^{\Nis}((U^m)_{|V-Z_{|V}})\ar[r]\ar[d] &L_{(\bcube,\loget)}\omega^*\cH^{\Nis}(V-Z_{|V})\ar[d]\\
        \colim_m L_{(\bcube,\loget)}\omega^*\cH^{\Nis}((U^m)_{|V})\ar[r]\ar[d] &L_{(\bcube,\loget)}\omega^*\cH^{\Nis}(V)\ar[d]\\
        (\colim_m L_{(\bcube,\loget)}\omega^*\cH^{\Nis}((\P^{d}/\P^{d-1})\times (U')^m)\ar[r]&L_{(\bcube,\loget)}\omega^*\cH^{\Nis}(\P^{d}/\P^{d-1})\times V').
    \end{tikzcd}
    \]
Since both $\omega^*$ and $L_{(\bcube,\loget)}$ are monoidal, we can take $\P^{d}/\P^{d-1}$ outside and by induction on dimension the bottom horizontal map is an equivalence, and the top horizontal map is an equivalence again by \cite[Lemma 6.4]{mericicrys} since $U_{|V-Z_{|V}}\to V-Z_{|V}$ is finite, so the middle horizontal map is an equivalence. We can then take $V\to X$ a trivializing cover and apply the diagram above to $V^r:= V\times_X V\times_X\ldots \times_X V$, so by Nisnevich descent of $\omega^*$ (see \cite[Proposition 3.7]{DoosungA1invlog}) we deduce that\[
\begin{aligned}
&\colim_{m} L_{(\bcube,\loget)}\omega^*\cH^{\Nis}((U^m)_{|X-Z_0})\simeq \colim_r \colim_{m} L_{(\bcube,\loget)}\omega^*\cH^{\Nis}((U ^m)_{|V^r})\\
&\simeq \colim_r L_{(\bcube,\loget)}\omega^*\cH^{\Nis}(V^r) \simeq L_{(\bcube,\loget)}\omega^*\cH^{\Nis}(X-Z_0).
\end{aligned}
\]  
By taking now $Z_n$ as the singular locus of $Z_{n-1}$, we get a chain of strict closed subschemes such that \[
\colim L_{(\bcube,\loget)}\omega^*\cH^{\Nis}((U^\bullet)_{|X-Z_n})\simeq L_{(\bcube,\loget)}\omega^*\cH^{\Nis}(X-Z_n)\textrm{ for all }n,
\] 
so we conclude since $X$ has finite Krull dimension. 
\end{proof}
For the rest of the section, we assume that $k$ satisfies resolutions of singularities as in \ref{nota:RS}. We are now ready to prove our main theorem:
\begin{thm}\label{thm:comparison}
    The adjunction\[
    \begin{tikzcd}
    \SH^{S^1}(k) \ar[r,"\omega^*", shift left = 2] &\logSH^{S^1}(k)\ar[l,"\omega_*"]
    \end{tikzcd}
    \]
    induces an adjunction\[
    \begin{tikzcd}
    \SH^{S^1}_t(k) \ar[r,"\omega^*_t", shift left = 2] &\logSH^{S^1}_{\loget}\ar[l,"\omega_*^t"](k)
    \end{tikzcd}
    \]
fitting in commutative diagrams:\[
    \begin{tikzcd}
    \SH^{S^1}(k) \ar[r,"\omega^*"]\ar[d,"L_{(\A^1,t)}"] &\logSH^{S^1}(k)\ar[d,"L_{(\bcube,\loget)}"] && 
    \SH^{S^1}(k) &\logSH^{S^1}(k)\ar[l,"\omega_*"]\\
     \SH^{S^1}_t(k) \ar[r,"\omega^*_t"] &\logSH^{S^1}_{\loget}(k) &&
    \SH^{S^1}_t(k) \ar[u,"i_{(\A^1,t)}"]&\logSH^{S^1}_{\loget}\ar[l,"\omega_*^t"](k)\ar[u,"i_{(\bcube,\loget)}"]
    \end{tikzcd}
\]
and similarly for $\DMeff$. The functor $\omega^*_t$ is monoidal.
\begin{proof}
    The functor $L_{(\bcube,\loget)}\omega^*$ satisfies $\A^1$-invariance by construction and tame descent by Theorem \ref{lem:tame-descent} above, therefore it factors through $L_{(\A^1,t)}$ inducing $\omega^*_t$ that fits in the first commutative square. Since $L_{(\bcube,\loget)}\omega^*$ is a left adjoint, it preserves all colimits. Since both $i_{(\A^1,t)}$ and $i_{(\bcube,\loget)}$ are fully faithful, we have that for $\{M_i\}_{i\in I}$ a system in $\DMeff[t](k,\Z)$:
    \begin{align*}
            &\colim_{i\in I} \omega^*_t M_i \simeq \colim_{i\in I} \omega^*_tL_{(\A^1,t))}i_{(\A^1,t)}M_i \simeq  
            \colim L_{(\bcube,\loget))}\omega^*i_{(\A^1,t)} M_i \simeq 
            L_{(\bcube,\loget))}\omega^* \colim i_{(\A^1,t)}M_i \\
            &\simeq \omega^*_t L_{(\A^1,t))}\colim i_{(\A^1,t)} M_i \simeq \omega^*_t \colim_{i\in I} M_i,
    \end{align*}
    where the last equivalence follows from the fact that $L_{(\A^1,t)}\dashv i_{(\A^1,t)}$ is a localization, so by the adjoint functor theorem there exists a right adjoint $\omega_*^t$, which fits by adjunction in the right square. The monoidality of $\omega^*_t$ then follows from the monoidality of $L_{(\bcube, \loget)}$ and $L_{(\A^1, t)}$, which holds by definition of the monoidal structure on tame motives, and the monoidality of $\omega^*$, which holds by \cite[Theorem 4.4(4)]{DoosungA1invlog}.
\end{proof}
\end{thm}
\begin{cor}
    The integral $p$-adic cohomology of \cite{mericicrys} factors through $\DMeff[t](k)$, inducing a realization\[
    R\Gamma_p^t\colon \DMeff[t](k)\to \cD(R(k))\quad X\mapsto \lim_m R\Gamma((\ol{X},\partial X), W_m\Lambda^\bullet),
    \]
    where $W_m\Lambda^\bullet$ is the logarithmic de Rham--Witt complex of Hyodo--Kato (see \cite{mokrane} or \cite{matsuue}) and $R(k)$ is the Raynaud ring of Ekhedal \cite{EkedahlII}.
    \begin{proof}
        $R\Gamma_p$ is defined as:\[
        \lim_m \Map_{\logDMeff[\loget](k,\Z)}(L_{(\bcube,\loget)}\omega^*(-),W_m\Lambda^\bullet)\colon \DMeff(k)\to \cD(R(k))
        \]
        and so by Theorem \ref{thm:comparison} this agrees with\[
        \DMeff(k)\xrightarrow{L_{(\A^1,t)}}  \DMeff[t](k)\xrightarrow{\lim_m R\Gamma(\omega^*_t(-), W_m\Lambda^\bullet)} \cD(R(k)).
        \]
    \end{proof}
\end{cor}

\begin{cor}
Let $F$ be a strictly $(\bcube,\loget)$-invariant sheaf of abelian groups, so that $F[0]\in \logSH^{S^1}_{\loget}$. Let $G\in \Sh_{t}(\Sm_k,\Spt)$ be the underlying sheaf of spectra of $\omega_*^t F[0]$. Then the tame sheaf $\pi_0G$ is strongly $\A^1$-invariant, \emph{i.e.} for all $X\in \Sm_k$\[
\pi_0G(\A^1_X)\simeq \pi_0G(X)\textrm{ and }H^1_t(\A^1_X,\pi_0G)\simeq H^1_t(X,\pi_0G).
\] 
\begin{proof}
    Since $\omega_*$ is left $t$-exact with respect to the homotopy $t$ structures of $\logSH^{S^1}(k)$ and $\SH^{S^1}(k)$ and the inclusion $R\epsilon_*\colon\Sh_t(\Sm_k,\Spt)\to \Sh_{\Nis}(\Sm_k,\Spt)$ is fully faithful and left $t$-exact for the usual $t$-structure induced by the Postnikov $t$-structure on $\Spt$, we deduce that $\tau_{\geq 1}G = 0$, so we have a fiber sequence in $\Spt$:\[
    \pi_0G(Y)\to G(Y)\to \tau_{< 0}G(Y),
    \] 
    which implies that $\pi_0G(Y) = \pi_0(G(Y))$. Since $G$ is $(\A^1,t)$-local, we conclude that $\pi_0(G(\A^1_X)) = \pi_0(G(X))$. Finally, consider the hypercohomology speactral sequence, which is functorial in $Y$:\[
    H^p_t(Y,\pi_{-q}G)\Rightarrow H^{p+q}_t(X,G)
    \]
    whose five term exact sequence gives an injective map\[
    H^1_t(Y,\pi_0G)\hookrightarrow H^1_t(Y,G).
    \]
    Since $G$ is $(\A^1,t)$-local, we have that $H^1_t(\A^1_X,G) = H^1_t(X,G)$, so the map $H^1_t(\A^1_X,\pi_0G) \to H^1_t(X,\pi_0G)$ induced by the zero section is injecitve. On the other hand, this map has a retraction induced by the projection $\A^1_X\to X$, so it is also surjective.
\end{proof}
\end{cor}

\begin{example}\label{ex:drW}
    Consider the \'etale sheaves $\nu_m(n)$: they fit in an left exact sequence in $\mathbf{RSC}_{\Nis}$:\[
    0\to \nu_m(n)\to W_m\Omega^n\xrightarrow{F-1}W_m\Lambda^n/dV^{m-1}\Omega^{n-1}
    \] 
    so by \cite{shujilog} and \cite[Theorem 4.4]{mericicrys} they fit in a left exact sequence of log \'etale sheaves:\[
        0\to \widetilde{\nu_m(n)}\to W_m\Lambda^n \xrightarrow{F-1}W_m\Lambda^n/dV^{m-1}\Lambda^{n-1}
    \]
    where $\widetilde{\nu_m(n)}$ is as in \eqref{eq:nu-tilde}.
    By \cite[Proposition 2.13]{lorenzon}, the last map is surjective in the log \'etale topology. Moreover, $W_m\Lambda^n$ and $W_m\Lambda^n/dV^{m-1}\Lambda^{n-1}$ are both  strictly $(\bcube,\loget)$-invariant by \cite[Theorem 4.2]{mericicrys}, hence $\widetilde{\nu_m(n)}$ are strictly $(\bcube,\loget)$-invariant, which implies that $L_{\loget}\omega^*\nu_m(n)\simeq \widetilde{\nu_m(n)}[0]$ is $(\bcube,\loget)$-local. This implies that $\pi_0\omega_*^tL_{\loget}\omega^*\nu_m(n)$ is strongly $(\A^1,t)$-invariant, but for all $X\in \Sm_k$ with smooth log compactification $(\ol{X},\partial X)$ as in Remark \ref{rmk:comp} we have that:
    \[
        \pi_0\omega_*^tL_{\loget}\omega^*\nu_m(n)(X) = \widetilde{\nu_m(n)}(X,\partial X) = \nu_m(n)(X),
    \]
    so we deduce that the sheaves $\nu_m(n)$ are strongly $(\A^1,t)$-invariant.
\end{example}

\section{\texorpdfstring{$\A^1$}{A1}-invariance of higher tame cohomology}

In this section, we use the result of \cite{amine} to prove the $\A^1$-invariance of higher tame cohomology. Notice that the assumption on resolutions of singularities on $k$ are still needed. 

\begin{lemma}\label{lem:pbf}
For all $m,n$, there is an equivalence in $\cD(\Sh_t^{\rm tr}(\Sm_k,\Z))$:\[
\nu_m(n)[0]\oplus \nu_m(n-1)[-1]\simeq R\uHom_{\Sh_{t}^{\tr}}(\Z_{\tr}(\P^1),\nu_m(n)).\] 
\end{lemma}
\begin{proof}
    Recall that $\nu_m(n)$ are $\A^1$-invariant Nisnevich sheaves with transfers and $(\nu_m(n))_{-1}\cong \nu_m(n-1)$, so the projective bundle formula in $\DMeff(k,\Z)$ (see \cite[Theorem 15.1 and Proposition 24.8]{MVW}) gives by adjunction maps in $\cD(\Sh_{\Nis}^{\rm tr}(\Sm_k,\Z))$\[
    (\nu_m(n)[0]\oplus \nu_m(n-1)[-1])\otimes \Z_{\rm tr}(\P^1) \to \nu_m(n).
    \]
    Since the tame sheafification is monoidal, again by adjunction this gives a map in $\cD(\Sh_{t}^{\rm tr}(k,\Z))$\[
     \nu_m(n)[0]\oplus \nu_m(n-1)[-1] \to R\uHom_{\Sh_t^{\rm tr}}(\Z_{\tr}(\P^1),\nu_m(n)).
    \]
    Let $X\in \Sm_k$: we show that the induced map\[
    R\Gamma(X,\nu_m(n))\oplus R\Gamma(X,\nu_m(n-1))[-1]\to R\Gamma(\P^1_X,\nu_m(n))
    \] 
    is an equivalence. Let $X\in \Sm_k$ and let $\ol{X}$ smooth and proper with $j\colon X\hookrightarrow \ol{X}$ open and $\ol{X}-j(X)$ supported on a simple normal crossing divisor $D=D_1+\ldots D_t$. Then we proceed by double induction on $\dim(X)$ and $t$. If $\dim(X)=0$ or $t=0$, then $X$ is proper so the map above is an equivalence by \cite[Proposition 8.2]{HS2021} and the projective bundle formula for the \'etale cohomology of $\nu_m(n)$ by \cite{Gros1985}. In general, let $X':= \ol{X}-(D_1\cup \ldots D_{t-1})$ and $D':= D_t-(D_1\cup \ldots D_{t-1})$: by \cite[Theorem 1.3.1.]{amine} we have a map of long exact sequences\[
    \begin{tikzcd}
    H^{q-2}_t(D',\nu_m(n-2))\oplus H^{q-1}_t(D',\nu_m(n-1))\ar[r,"(1)"]\ar[d]&H^{q-1}_t(\P^1_{D'},\nu_m(n-1))\ar[d]\\
        H^{q-1}_t(X',\nu_m(n-1))\oplus H^{q}_t(X',\nu_m(n)) \ar[r,"(2)"]\ar[d]&H^q_t(\P^1_{X'},\nu_m(n))\ar[d]\\
        H^{q-1}_t(X,\nu_m(n-1))\oplus H^q_t(X,\nu_m(n))\ar[r,"(3)"]\ar[d]&H^q_t(\P^1_X,\nu_m(n))\ar[d]\\
        H^{q-1}_t(D',\nu_m(n-2))\oplus H^{q}_t(D',\nu_m(n-1))\ar[r,"(4)"]\ar[d]&H^{q}_t(\P^1_{D'},\nu_m(n-1))\ar[d]\\
        H^{q}_t(X',\nu_m(n-1))\oplus H^{q+1}_t(X',\nu_m(n)) \ar[r,"(5)"]&H^{q+1}_t(\P^1_{X'},\nu_m(n))
    \end{tikzcd}
    \]
    By induction hypotheses the maps $(1),(2),(4),(5)$ are isomorphisms, so $(3)$ is also an isomorphism.
\end{proof}
\begin{cor}\label{cor:pbf-tamely-hens}
    Let $X$ be the tame henselization of a smooth scheme over $k$ at a tame point $(\ol{x},v)$. Then $H^q_t(\P^1_X,\nu_m(n))=0$ for $q\geq 2$. 
    \begin{proof}
        We have $H^q_t(\P^1_X,\nu_m(n))\cong H^q_t(X,\nu_m(n))\oplus H^{q-1}_t(X,\nu_m(n-1))$ by Lemma \ref{lem:pbf} and \cite[Theorem 4.7]{HS2021}, and the right hand side is zero for $q\geq 2$ since $X$ is tamely henselian.
    \end{proof}
\end{cor}
\begin{thm}\label{thm:A1-inv-higher-cohomology}
    Let $X$ be the tame henselization of a smooth scheme over $k$ at a tame point $(\ol{x},v)$. Let $U\subseteq \P^1_{\ol{k}}$ be the complement of finitely many closed points. Then $H^q_t(U_X,\nu_m(n))=0$ for $q\geq 2$. 
    \begin{proof}
        Let $U=\P^1_{\ol{k}}-\{x_1\ldots x_n\}$. Since $X=\Spec(A)$ is tamely henselian with respect to $v$ trivial on $k$, we have that $\ol{k}\hookrightarrow A$, therefore $U_X \cong \P^1_{\ol{k}}-\{X_1\ldots X_n\}$, where $X_i=X\times_{\ol{k}}{x_i}$ is isomorphic to $X$, hence it is tamely henselian. Since $X$ is a cofiltered limit of smooth schemes with affine transition morphisms, \cite[Theorem 1.3.1]{amine} and \cite[Theorem 4.6]{HS2021} give a long exact sequence\[
        \ldots \to H^q_t(\P^1_X,\nu_m(n))\to H^q_t(U_X,\nu_m(n))\to \oplus H^{q}(X_i,\nu_m(n-1))\to \ldots
        \]
        Then $\oplus H^{q}_t(X_i,\nu_m(n-1))=0$ since each $X_i$ is tamely henselian and $H^q_t(\P^1_X,\nu_m(n))=0$ by Corollary \ref{cor:pbf-tamely-hens}, so the theorem follows.
    \end{proof}
\end{thm}

\section{Proof of the main theorems}

In this section, we put together the results and prove the theorems in the introduction. 
Let $k$ be a field that satisfies resolutions of singularities as in \ref{nota:RS}. Then combining Example \ref{ex:drW} and Theorem \ref{thm:A1-inv-higher-cohomology}, we conclude that for all $X$ tame henselization of a smooth scheme at a tame point $(\ol{x},v)$ and all $q\geq 0$ we have
\begin{equation}\label{eq:A1-inv}
    H^q_t(\A^1_X,\nu_m(n))\simeq H^q_t(X,\nu_m(n))
\end{equation}
We are now ready to prove Theorem \ref{thm:motivic-dRW}, \emph{i.e.} the existence of the motivic ring spectrum $H\Z/p^m$. By the usual spectral sequence argument \eqref{eq:A1-inv} implies that $\nu_m(n)$ is a strictly $\A^1$-invatiant tame sheaf with transfers, in particular the collection $\nu_m(*)[-*]$ gives rise to a graded commutative monoid in $\DMeff[t](k)$. Then to conclude it is enough to show that the maps induced by the Chern classes\[
R\Gamma_t(X,\nu_m(n))\oplus R\Gamma_t(X,\nu_m(n-1))[1]\to R\Gamma_t(\P^1_X,\nu_m(n))
\]
are equivalences, which follows from Lemma \ref{lem:pbf}. Then the spectrum $H\Z/p^m$ exists by Theorem \ref{thm:ring-spectra}
\begin{rmk}\label{rmk:GL}
In fact, by Geisser--Levine \cite[Theorem 1.1]{GL}, there is a quasi-isomorphism of complexes of Nisnevich sheaves with transfers: $\Z/p^m(n)\simeq \nu_m(n)[-n]$, where the left hand side is the motivic complex. This implies that the graded $E_{\infty}$-ring in $\DMeff[\Nis](k,\Z/p^m)$ given by the collection $\nu_m(*)[-*]$ builds up the unit of $\DM[\Nis](k,\Z/p^m)$. Since the localization functor $L_{(\A^1,t)}^{st}\colon \DM[\Nis](k,\Z/p^m)\to \DM[t](k,\Z/p^m)$ is monoidal, it preserves the unit: this implies that the ring spectrum $H\Z/p^m$ is the unit of $\DM[t](k,\Z/p^m)$, since for all $m,n$ we have that $L_{(\A^1,t)}\nu_m(n) = L_{t}\nu_m(n)$.
\end{rmk}
As observed in the introduction, Theorem \ref{thm:main-intro} is now a mere consequence of Theorem \ref{thm:mot-intro} and the motivic properties of $\DM$ (see \eqref{item:1} and \eqref{item:2}).
Finally, we prove Theorem \ref{thm:comparison}, which we restate for the sake of the reader:
\begin{thm}
Let $k$ be a perfect field of characteristic $p$ that satisfies resolutions of singularities as in \ref{nota:RS}. For all $X\in \Sm_k$ with smooth log compactification $(\ol{X},\partial X)$ as in Remark \ref{rmk:comp}, we have that\[
R\Gamma_t(X,\nu_m(n))\simeq R\Gamma_{\loget}((\ol{X},\partial X),\widetilde{\nu_m(n)})
\]
\begin{proof}    
The canonical map $\nu_m(n)[0]\to \omega_*^t\omega^*_t\nu_m(n)[0]$ induces a map
\begin{align*}
R\Gamma_t(X,\nu_m(n))\to &\Map_{\DM[t]}(\cM^t(X), \omega_*^t\omega^*_t\nu_m(n)[0])\simeq \\
&\Map_{\logDM[\loget]}(\cM^{\loget}(\ol{X},\partial X),L_{(\bcube,\loget)}\omega^*\nu_m(n)[0])
\end{align*}
functorial in $X$ and $(\ol{X},\partial X)$, and since $L_{\loget}\omega^*\nu_m(n)[0]\simeq \widetilde{\nu_m(n)}[0]$ is already $(\bcube,\loget)$-local, we have that $L_{(\bcube,\loget)}\omega^*\nu_m(n)[0]\simeq \widetilde{\nu_m(n)}[0]$, therefore the map above induces a map \[R\Gamma_t(X,\nu_m(n))\to R\Gamma_{\loget}((\ol{X},\partial X),\widetilde{\nu_m(n)}).\] We will show that this map is an equivalence. 
Let $|\partial X| = D_1 + \ldots D_r$: we proceed by double induction on $\dim(X)$ and $r$. If $r=0$, then $X=\ol{X}$ is proper, so\[
H^q_t(X,\nu_m(n))\simeq H^q_{\et}(X,\nu_m(n)) = H^q_{\loget}((X,\triv),\widetilde{\nu_m(n)}).
\]
If $\dim(X)=0$, then $\partial X=0$ so it follows from the case above. For $\dim(X)>0$ and $r(X)>0$, let $\partial X'$ be the log structure supported on $D_1+\ldots D_{r-1}$ and $X':=\ol{X}-|\partial X'|$: then Morel--Voevodsky purity of \ref{item:2} in \ref{rmk:pbf-gysin} gives the fiber sequence \[
\cM^t(X)\to \cM(X')\to \cM^t(D_1\cap X')\otimes \cM^t(\P^1,i_0),
\]
so using the computation of the cohomology of $\P^1$ in Lemma \ref{lem:pbf} and \cite[Corollary 4.5]{mericicrys} gives for all $q$ a diagram whose columns are long exact sequences:\[
\begin{tikzcd}
    H^{q-1}_t(D_1\cap X',\nu_m(n-1))\ar[r,"(1)"]\ar[d]& H^{q-1}_{\loget}((D_1\partial X_{D_1}),\widetilde{\nu_m(n-1)})\ar[d]\\
    H^q_t(X',\nu_m(n))\ar[r,"(2)"]\ar[d]&H^q_{\loget}((X,\partial X'),\widetilde{\nu_m(n)})\ar[d] \\
    H^q_t(X,\nu_m(n))\ar[r,"(3)"]\ar[d]& H^q_{\loget}((X,\partial X),\widetilde{\nu_m(n)})\ar[d]\\
    H^q_t(D_1\cap X',\nu_m(n-1))\ar[r,"(4)"]\ar[d]& H^q_{\loget}((D_1\partial X_{D_1}),\widetilde{\nu_m(n-1)})\ar[d]\\
    H^{q+1}_t(X',\nu_m(n))\ar[r,"(5)"] &H^{q+1}_{\loget}((X,\partial X'),\widetilde{\nu_m(n)})\\
    &&&&
\end{tikzcd}
\]
By induction on the dimension, $(1)$ and $(4)$ are isomorphisms, and by induction on $r$ $(2)$ and $(5)$ are isomorphisms, so $(3)$ is an isomorphism, concluding the proof.
\end{proof}
\end{thm}

\bibliographystyle{alpha}
\bibliography{bibMerici}

\end{document}